\documentclass{gtpart}     
\gtart  
\usepackage{pinlabel}  
\usepackage{graphicx}  
\title[Pseudo-Anosov maps not arising from Penner's construction]{Pseudo-Anosov mapping classes not arising from Penner's construction}

\author{Hyunshik Shin}
\givenname{Hyunshik}
\surname{Shin}
\address{Department of Mathematics, Statistics, and Computer Science\\
University of Illinois at Chicago\\\newline
Chicago, IL 60607\\USA}
\email{shin@math.uic.edu}
\urladdr{http://www.math.uic.edu/~shin}

\author{Bal\'azs Strenner}
\givenname{Bal\'azs}
\surname{Strenner}
\address{Department of Mathematics\\University of
  Wisconsin--Madison\\\newline
  Madison, WI 53705\\USA
}
\email{strenner@math.wisc.edu}
\urladdr{http://math.wisc.edu/~strenner}


\arxivreference{math.GT/1410.6974}
\arxivpassword{}

\volumenumber{}
\issuenumber{}
\publicationyear{}
\papernumber{}
\startpage{}
\endpage{}
\doi{}
\MR{}
\Zbl{}
\received{}
\revised{}
\accepted{}
\published{}
\publishedonline{}
\proposed{}
\seconded{}
\corresponding{}
\editor{}
\version{}

\newtheorem{thm}{Theorem}[section]
\newtheorem*{conj}{Conjecture (Penner, 1988)}
\newtheorem*{pennerconst*}{Penner's Construction}
\newtheorem*{mainthm*}{Main Theorem}
\newtheorem{lemma}[thm]{Lemma}
\newtheorem{cor}[thm]{Corollary}
\newtheorem{prop}[thm]{Proposition}
\newtheorem{question}[thm]{Question}

\newcommand{\N}{\mathbb{N}}
\renewcommand{\v}{\mathbf{v}} 
\newcommand{\F}{\mathcal{F}}
\newcommand\tS{\widetilde{S}}

\newcommand\SL{\mathrm{SL}}
\newcommand\Mod{\mathrm{Mod}}
\newcommand\PSL{\mathrm{PSL}}
\newcommand\mC{\mathcal{C}}
\newcommand\Isom{\mathrm{Isom}}

\begin{document}

\begin{abstract}    
  We show that Galois conjugates of stretch factors of pseudo-Anosov
  mapping classes arising from Penner's construction lie off the unit
  circle. As a consequence, we show that for all but a few exceptional
  surfaces, there are examples of pseudo-Anosov mapping classes so
  that no power of them arises from Penner's construction. This
  resolves a conjecture of Penner.
\end{abstract}

\maketitle


\section{Introduction}

Let $S_{g,n}$ be the orientable surface of genus $g$ with $n$
punctures. The mapping class group $\Mod(S_{g,n})$ is the group of
isotopy classes of orientation-preserving homeomorphisms of $S_{g,n}$.
Thurston's classification theorem \cite{Thurston88} states that each
element of $\Mod(S_{g,n})$ is either periodic, reducible, or
pseudo-Anosov. An element $f \in \Mod(S_{g,n})$ is pseudo-Anosov if
there is a representative homeomorphism $\psi$, a number $\lambda>1$,
and a pair of transverse invariant singular measured foliations $\F^u$
and $\F^s$ such that
\begin{displaymath}
  \psi(\F^u) = \lambda\F^u \quad \mbox{and} \quad \psi(\F^s) =
  \lambda^{-1}\F^s.
\end{displaymath}
The number $\lambda$ is called the stretch factor (or dilatation) of $f$.

Isotopy classes of orientation-preserving Anosov maps of the torus can
easily be classified as actions of matrices $M \in \SL(2,\Z)$ with
$|\mathrm{tr}(M)|>2$ on $\R^2/\Z^2$. However, it is much harder to
give explicit examples of pseudo-Anosov maps on more complicated
surfaces.

Thurston gave the first general construction of pseudo-Anosov mapping
classes in terms of Dehn twists \cite{Thurston88}. After Thurston's
work, various other constructions have been developed
\cite{ArnouxYoccoz81, Kra81, Long85, Penner88, CassonBleiler88,
  BestvinaHandel92}. In this paper, we study Penner's construction
\cite{Penner88}.

\begin{pennerconst*}\label{thm:penners_construction}
  Let $A = \{a_1, \ldots, a_n\}$ and $B = \{b_1, \ldots, b_m\}$ be a
  pair of multicurves on a surface $S$. Suppose that $A$ and $B$ are
  filling, that is, $A$ and $B$ are in minimal position and the
  complement of $A\cup B$ is a union of disks and once punctured disks.
  Then any product of positive Dehn twists about $a_j$ and negative
  Dehn twists about $b_k$ is pseudo-Anosov provided that all $n+m$
  Dehn twists appear in the product at least once.
\end{pennerconst*}

Penner \cite{penner1991bounds} used this construction to give examples
of pseudo-Anosov mapping classes with small stretch factors. (See also
\cite{bauer1992upper} and \cite{Leininger04} for more work on small
stretch factors arising from Penner's construction.)

Pseudo-Anosov maps arising from Penner's construction fix the
singularities and separatrices of their invariant foliations, and
therefore not all pseudo-Anosov mapping classes arise from Penner's
construction. However, since the construction is fairly general,
Penner conjectured the following.

\begin{conj}\label{conj:penners_conjecture}
  Every pseudo-Anosov mapping class has a power that arises from Penner's construction.
\end{conj}

The conjecture is listed as Problem 4 in Chapter 7 of \cite{Farb06}
and also discussed briefly in Section 14.1.2 of \cite{FarbMargalit12}.

It is a folklore theorem that Penner's construction is true for
$S_{1,0}$ and $S_{1,1}$ and that it is false for $S_{0,4}$, but a
modified version of the conjecture, allowing half-twists in addition
to Dehn twists in Penner's construction, is true. To the best of our
knowledge, no proof of this has appeared in the literature. In the
appendix, we give a proof by considering the action of the mapping
class group on the curve complex. The main result of this paper is the
answer to Penner's conjecture in the remaining nontrivial cases.

We call a pseudo-Anosov mapping class and its stretch factor $\lambda$
\emph{coronal} if $\lambda$ has a Galois conjugate on the unit circle.

\begin{mainthm*}\label{thm:penner_false}
  A coronal pseudo-Anosov mapping class has no power coming from
  Penner's construction. Moreover, there exists a coronal
  pseudo-Anosov mapping class on $S_{g,n}$ when $3g+n \ge 5$. In
  particular, Penner's conjecture is false for $S_{g,n}$ when
  $3g+n \ge 5$.
\end{mainthm*}

We remark that even the modified version of the conjecture, allowing
half-twists in addition to Dehn twists in Penner's construction, is
false for $S_{0,n}$ when $n \ge 5$.

The proof of the first part of the Main Theorem is based on the
fact that stretch factors of pseudo-Anosov mapping classes arising
from Penner's construction appear as Perron--Frobenius eigenvalues of
products of certain integral matrices, which depend only on the
intersection numbers of curves. We show that such matrix products may
not have eigenvalues on the unit circle other than 1, which implies
that pseudo-Anosov stretch factors arising from Penner's construction
are not coronal. 

The key idea is that an eigenvalue on the unit circle
corresponds to a rotation on an invariant plane, which we consider a
dynamical system. Our topological setting provides a natural quadratic
form $h$ which, considered as a height function, plays a role similar
to that of Lyapunov functions in stability theory. We show that the
products of matrices arising from Penner's construction act by
increasing the height, which prohibits rotations on subspaces.

To prove the second part of the Main Theorem, we use known
coronal pseudo-Anosov mapping classes on $S_{2,0}$ and $S_{0,5}$ to
construct coronal pseudo-Anosov mapping classes on the rest of the
surfaces via introducing punctures and taking branched covers.

The Main Theorem provides a number-theoretical obstruction for
pseudo-Anosov maps to arise from Penner's construction: if the stretch
factor of $f$ has a Galois conjugate on the unit circle, then no power
of $f$ can arise from Penner's construction. We do not know whether
there are other obstructions.

\begin{question}
  Let $f$ be a pseudo-Anosov mapping class whose stretch factor does
  not have Galois conjugates on the unit circle. Does $f^n$ arise from
  Penner's construction for some $n \in \N$?
\end{question}

\paragraph{Acknowledgements} The authors are grateful to Richard Kent and
Dan Margalit for numerous helpful conversations and invaluable
comments. We also thank Richard Kent for suggesting the term
\emph{coronal} and the referee for many helpful comments.

\section{Proof of the Main Theorem}
\subsection{Stretch factors arising from Penner's construction}
Let $A =\{a_1,\ldots, a_n\}$ and $B = \{b_1,\ldots,
b_m\}$ be a pair of multicurves on a surface. Introduce the
notation
\begin{displaymath}
(e_1, \ \ldots \,, \ e_{n+m}) = (a_1, \ \ldots, \ a_n, \, b_1, \ \ldots, \ b_m).
\end{displaymath}
The \emph{intersection matrix}
of $A$ and $B$ is the symmetric $(n+m)\times (n+m)$ nonnegative
integral matrix $\Omega = \Omega(A,B)$ whose $(j,k)$-entry is the 
geometric intersection number $i(e_j,e_k)$.

\paragraph{The monoid $\Gamma(\Omega)$}
Penner showed that actions of the Dehn twists $T_{a_j}$ and $T_{b_k}^{-1}$ on
$A \cup B$ can be described by the matrices 
\begin{displaymath}
  Q_i = I + D_i\Omega \quad (1\le i \le n+m),
\end{displaymath}
where $I$ is the $(n+m)\times (n+m)$ identity matrix, and $D_i$
denotes the $(n+m)\times (n+m)$ matrix whose $i$th entry on the
diagonal is 1 and whose other entries are zero. 
Any product of $T_{a_j}$ and $T_{b_k}^{-1}$, where each $a_j$ and each $b_k$
appear at least once, is pseudo-Anosov,
and its stretch factor is given by the Perron--Frobenius eigenvalue of the
corresponding product of the matrices $Q_i$. Therefore one can study
pseudo-Anosov stretch factors arising from Penner's construction by
studying the monoid
\begin{displaymath}
\Gamma(\Omega) = \langle Q_i : 1\le i \le n+m\rangle,
\end{displaymath}
generated by the matrices $Q_i$ depending on $\Omega$. For more
details, see \cite{Penner88}.

\paragraph{The height function $h$} Define the quadratic form
$h\co \R^{n+m}\to\R$ by the equation
\begin{displaymath}
  h(\v) = \frac12 \v^T \Omega \v.
\end{displaymath}
Geometrically, the vector $\v$ corresponds to assigning a real number
to each curve in $A$ and $B$. The function $h$ is the sum of the
products of the values of intersecting curves over all intersection
points.

The multicurves $A$ and $B$ define two transverse cylinder
decompositions of the surface. When $\v>0$, the values assigned to the
curves can be thought of as the widths of the cylinders. This way we
get a singular flat metric on the surface with a rectangle
corresponding to each intersection, and the area of this flat surface
is $h(\v)$. When $\v$ is not positive, one can still think of $h(\v)$
as a signed area. However, it is not clear how this geometric
interpretation explains the following interaction between the function
$h(\v)$ and the matrices $Q_i$.

\begin{prop}
  $h(Q_i\v) - h(\v) = ||Q_i\v-\v||^2$.
\end{prop}
\begin{proof}
  Since all entries on the diagonal of $\Omega$ are
  zero, we have $D_i\Omega D_i = 0$ for all $i$, and hence we have
  \begin{displaymath}
    \frac12 Q_i^T\Omega Q_i - \frac12 \Omega = \frac12 (I+\Omega D_i) \Omega (I + D_i\Omega)
    -\frac12 \Omega = \Omega D_i \Omega.
  \end{displaymath}
  It follows that
  $h(Q_i\v) - h(\v) = ||D_i\Omega\v||^2 = ||Q_i\v-\v||^2$.
\end{proof}

\begin{cor}\label{cor:increasing}
  If $M\in \Gamma(\Omega)$, then $h(M\v) \ge h(\v)$ with equality if
  and only if 
  $M\v = \v$. 
\end{cor}

\begin{prop}\label{prop:Gamma_unit_circle}
  If $M\in \Gamma(\Omega)$, then $M$ cannot have eigenvalues on the
  unit circle except 1.
\end{prop}
\begin{proof}
  Assume for contradiction that $M$ has an eigenvalue $\mu\ne 1$ on
  the unit circle. Then there exists $\v\in \R^{n+m}$ and a sequence
  $p_i \to \infty$ of positive integer powers such that $M\v \ne \v$
  and $M^{p_i}\v \to \v$. (If $\mu\ne -1$, choose $\v$ to be any
  nonzero vector in the two-dimensional invariant subspace on which
  $M$ acts by a rotation. If $\mu= -1$, choose $\v$ to be a
  corresponding eigenvector.) Therefore we have $h(M\v) > h(\v)$, and
  hence $h(M^{p_i}\v) \ge h(M(\v))$ for all $p_i$ by Corollary
  \ref{cor:increasing}. However, we have
  $h(M^{p_i}\v) \rightarrow h(\v)$ by continuity, which is a
  contradiction.
\end{proof}

\subsection{Coronal pseudo-Anosov mapping classes} Recall that a
pseudo-Anosov mapping class and its stretch factor $\lambda$ are
\emph{coronal} if $\lambda$ has a Galois conjugate on the unit
circle.

\begin{lemma}\label{lemma:powers}
  If a pseudo-Anosov mapping class $f$ is coronal, then each power
  of $f$ is also coronal.
\end{lemma}
\begin{proof}
  Let $\lambda$ be the stretch factor of $f$.  Let $\sigma$ be an
  automorphism of the Galois extension $L/\Q$ with
  $|\sigma(\lambda)|=1$, where $L$ is the splitting field of the
  minimal polynomial of $\lambda$. For all $k\ge 1$,
  we have $|\sigma(\lambda^k)| = |\sigma(\lambda)^k| = 1$. Therefore
  $\lambda^k$, the stretch factor of $f^k$, has a Galois conjugate
  $\sigma(\lambda^k)$ on the unit circle.
\end{proof}

As a consequence of Proposition \ref{prop:Gamma_unit_circle} and Lemma
\ref{lemma:powers}, we have the following.

\begin{cor}[First part of the Main Theorem]\label{cor:nopower}
A coronal pseudo-Anosov mapping class has no power coming from
Penner's construction.
\end{cor}

To complete the proof of the Main Theorem, we need to show that
coronal pseudo-Anosov mapping classes exist on all but a few
exceptional surfaces. 

\begin{lemma}\label{lemma:covering}
  If there exists a coronal pseudo-Anosov mapping class on a surface
  $S$, and there is a branched covering $\widetilde{S}\to S$, then
  there exists a coronal pseudo-Anosov mapping class on $\tS$ as well.
\end{lemma}
\begin{proof}
  If $f \in \Mod(S)$ is a coronal pseudo-Anosov mapping class, then
  some power of $f$ can be lifted to a pseudo-Anosov mapping class
  $\widetilde{f}$ on $\tS$ with the same stretch factor as the power
  of $f$ (see \cite[Expos\'e 13,
  II.1.]{FLP}
  ). By Lemma \ref{lemma:powers}, $\widetilde{f}$ is also coronal.
\end{proof}

\begin{lemma}\label{lemma:puncturing}
  If there exists a coronal pseudo-Anosov mapping class on a surface
  $S$, then there exists a coronal pseudo-Anosov mapping class on the
  surface $S' = S\setminus\{p\}$ with one more puncture as well.
\end{lemma}
\begin{proof}
  Let $f \in \Mod(S)$ be a coronal pseudo-Anosov mapping class with
  stretch factor $\lambda$ and let $\psi$ be its representative
  homeomorphism. Some power $\psi^k$ has a fixed point $p$ (see
  \cite[Proposition 9.20]{FLP} or \cite[Theorem
  14.19]{FarbMargalit12}), and hence $\psi^k$ induces a pseudo-Anosov
  homeomorphism of $S\setminus\{p\}$ with coronal stretch factor
  $\lambda^k$.
\end{proof}

\begin{prop}[Second part of the Main Theorem]
  \label{prop:constructions}
  There exists a coronal pseudo-Anosov mapping class on $S_{g,n}$ when
  $3g+n \geq 5$.
\end{prop}
\begin{proof}
  On $S_{2,0}$ there is a coronal pseudo-Anosov mapping class with
  stretch factor the Perron root of the polynomial $x^4-x^3-x^2-x+1$
  \cite{Zhirov95}. For each $g\ge 3$ there is an unbranched covering
  of $S_{2,0}$ by $S_{g,0}$. It follows from Lemma
  \ref{lemma:covering} and Lemma \ref{lemma:puncturing} that there
  exists a coronal pseudo-Anosov mapping class on all $S_{g,n}$ with
  $g\ge 2$ and $ n \geq 0$.

  For the genus 0 cases, start from a coronal pseudo-Anosov mapping
  class on $S_{0,5}$ with stretch factor the Perron root of
  $x^4-2x^3-2x+1$ \cite{LanneauThiffeault11Braids}. By Lemma
  \ref{lemma:puncturing}, there exists a coronal pseudo-Anosov mapping
  class on $S_{0,n}$ for each $n\ge 5$. 

  Finally, there is a branched covering $S_{1,2} \to S_{0,5}$, induced
  by the hyperelliptic involution of $S_{1,2}$ exchanging the two
  punctures, which yields a coronal pseudo-Anosov mapping class on
  $S_{1,2}$ by Lemma \ref{lemma:covering}. (Technically, here
  $S_{1,2}$ and $S_{0,5}$ should be considered surfaces with marked
  points, not punctures. Because the theory of pseudo-Anosov maps and
  stretch factors is the same on surfaces with punctures and on
  surfaces with marked points, we can go back and forth between marked
  points and punctures as is convenient.) By Lemma
  \ref{lemma:puncturing}, there exists a coronal pseudo-Anosov mapping
  class on $S_{1,n}$ for $n\ge 2$.
\end{proof}

The Main Theorem immediately follows from Corollary \ref{cor:nopower}
and Proposition \ref{prop:constructions}.

\section{Remarks on the Galois conjugates of stretch factors}

\paragraph{Examples of coronal pseudo-Anosov mapping classes}
The set of coronal pseudo-Anosov mapping classes is presumably much
larger than the set of examples constructed above. For example, the
minimal pseudo-Anosov stretch factors tend to be coronal. In fact,
when $g = 2,3,4,5,7,8$, the minimal stretch factor on $S_g$ among
pseudo-Anosov mapping classes with orientable foliations are known,
and they are all coronal \cite{LanneauThiffeault11}. The minimal
pseudo-Anosov stretch factors on the surfaces
$S_{0,n}$ for $5\le n\le 9$ are also all coronal with the exception of
$n=8$ \cite{LanneauThiffeault11Braids}.

Not only is the set of coronal pseudo-Anosov mapping classes infinite,
but so is the set of coronal stretch factors (even modulo taking
powers). This follows from the first author's examples of
pseudo-Anosov mapping classes on $S_{g}$ with stretch factor a degree
$2g$ Salem number \cite{Shin14}. Hironaka's infinite family of
pseudo-Anosov mapping classes coming from the fibration of a single
3--manifold \cite{Hironaka10} also seem to consist mostly of coronal
pseudo-Anosov mapping classes whose stretch factors can have
arbitrarily high algebraic degree.

The abundance of coronal pseudo-Anosov mapping classes are also
suggested by computer experiments of Nathan Dunfield and Giulio Tiozzo
on random walks in the group of braids with 10 and 14 strands. Using
the standard Artin generators, mean length 25, variance 9, and a
sample of 100,000 pseudo-Anosov mapping classes, 94\% of the stretch
factors had Galois conjugates on the unit circle. Computer experiments
also show that a random reciprocal polynomial is very likely to have a
root on the unit circle. This may suggest that pseudo-Anosov mapping
classes arising from Penner's construction are actually rare.

\paragraph{Location of Galois conjugates}
It would be interesting to know precise constraints on the location of
Galois conjugates of pseudo-Anosov stretch factors arising from
Penner's construction. In particular, we wonder if they can at least
approach the unit circle or if they are even dense in $\C$. A positive
answer would imply that Galois conjugates of all pseudo-Anosov stretch
factors are dense in $\C$, which is also suggested by the experiments
of Dunfield and Tiozzo.

\appendix

\section{Penner's conjecture for the exceptional surfaces}

In this appendix, we show that Penner's construction is true for
$S_{1,0}$ and $S_{1,1}$ and that it is false for $S_{0,4}$, but a
modified version of the conjecture, allowing half-twists in addition
to Dehn twists in Penner's construction, is true.

\paragraph{The curve complex} Let $S$ be one of these three surfaces. The
modified curve complex $\mC(S)$ is a graph with vertices the isotopy
classes of simple closed curves on $S$, where two vertices are
connected by an edge if they have minimal intersection number (one for
$S_{1,0}$ and $S_{1,1}$, and two for $S_{0,4}$). In all three cases,
$\mC(S)$ is isomorphic to the 1-skeleton of the Farey tessellation
$\F$ of the hyperbolic plane (Figure \ref{figure:farey}). For more
details, see \cite[Section 4.1.1]{FarbMargalit12}.

\paragraph{The action of $\Mod(S)$} Let us consider the action of
$\Mod(S)$ on $\mC(S)$, which gives rise to a homomorphism
\begin{displaymath}
A \co \Mod(S) \to \Isom^+(\F) \cong \PSL(2,\Z)
\end{displaymath}
of $\Mod(S)$ to the orientation-preserving isometries of $\F$, once an
identification of $\mC(S)$ with $\F$ is chosen. We denote the image of
an element $f \in \Mod(S)$ by $A_f$.

\paragraph{Actions of Dehn twists} We call an ideal triangle in the
complement of $\F$ a \emph{tile}. A rotation of $\F$ about a vertex
$v$ of $\F$ to the left by $k$ tiles is defined as the parabolic
element of $\Isom^+(\F)$ that fixes $v$ and shifts the tiles adjacent to
$v$ in counterclockwise direction by $k$. Rotations to the right are
defined analogously.

For a Dehn twist $T_c$ about a curve $c$ in $S$, the isometry $A_{T_c}$
is parabolic, and it fixes the vertex of $\F$ corresponding to
$c$. Depending on the choice of identification of $\mC(S)$ with $\F$,
positive Dehn twists can act by rotating $\F$ to the left or to the
right. We choose the identification so that positive Dehn twists
correspond to rotations to the right and negative Dehn twists
correspond to rotations to the left. Note that Dehn twists on
$S_{1,0}$ and $S_{1,1}$ act by rotations by one tile, but Dehn twists
on $S_{0,4}$ act by rotations by two tiles. It is the half-twists on
$S_{0,4}$ that correspond to rotations by one tile.

\begin{figure}[tp] 
\labellist
\pinlabel $a$ at 30 30
\pinlabel $b$ at 77 1
\pinlabel $e_0$ at 70 28
\pinlabel {$A_f(e) = e_5$} at 174 192
\pinlabel $\gamma$ at 240 190
\pinlabel {$e_1$} [ ] at 93 53
 \pinlabel {$e_2$} [ ] at 70 112
 \pinlabel {$e_3$} [ ] at 117 154
 \pinlabel {$e_4$} [ ] at 185 136
 \pinlabel {$t_1$} [ ] at 89 36
 \pinlabel {$t_2$} [ ] at 69 81
 \pinlabel {$t_3$} [ ] at 96 134
 \pinlabel {$t_4$} [ ] at 154 114
 \pinlabel {$t_5$} [ ] at 171 173
\endlabellist
\centering
\includegraphics[width=10cm]{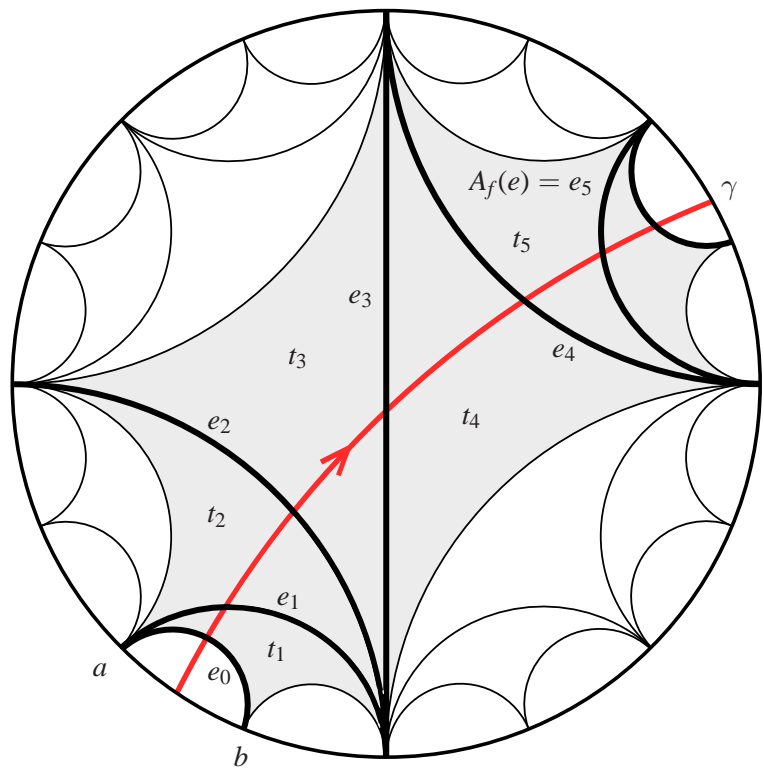}
\caption{The action $A_f$ of a pseudo-Anosov mapping class $f$ on the
  Farey tessellation. }
\label{figure:farey}
\end{figure}

\paragraph{Actions of pseudo-Anosov elements} For a pseudo-Anosov element
$f \in \Mod(S)$, the isometry $A_f$ is hyperbolic in $\PSL(2,\Z)$ and hence
$A_f$ has an invariant geodesic on the hyperbolic plane, called the 
\textit{axis} $\gamma$ of $A_f$. Since $f$ does not
fix any curve on $S$, $A_f$ does not fix any vertex of $\F$. In particular,
the endpoints of the axis $\gamma$ of $A_f$ are not vertices of $\F$.
Therefore $\gamma$ traverses a bi-infinite sequence of triangles in the
Farey tessellation, and it cuts two sides of each triangle.

Associated to $f$, there is a bi-infinite sequence of letters $L$ and
$R$ obtained as follows: travel along $\gamma$ in the direction of the
translation, and for each triangle record if the common vertex of the
cut sides are on the left or the right side of $\gamma$. This sequence is
periodic, because $A_f$ is a translation along $\gamma$. As the following
lemma shows, this bi-infinite sequence encodes how the hyperbolic
isometry $A_f$ can be written as a composition of parabolic
isometries.

\begin{lemma}\label{lemma:isometry_as_product}
  Let $e_0$ be an edge of $\F$ intersecting $\gamma$. Let $a$ and $b$ be
  the endpoints of $e_0$ on the left and right hand side of $\gamma$,
  respectively. Let $e_1, e_2, \ldots, e_n = A_f(e_0)$ be edges of $\F$
  intersected by $\gamma$ such that $e_{k-1}$ and $e_k$ are different sides
  of an ideal triangle $t_k$ of $\F$ for all $1\le k \le n$. (See
  Figure \ref{figure:farey} for an illustration when $n=5$.) For all
  $1\le k \le n$, define $s_k$ to be the letter $L$
  or the letter $R$ depending on whether the common vertex of
  $e_{k-1}$ and $e_k$ is on the left or right side of $\gamma$.

  Let $\tau_a$ and $\tau_b$ be the rotations of $\F$ by one tile to
  the right about the points $a$ and $b$, respectively, and introduce
  the notation
  \begin{displaymath}
    \tau(s) =
    \begin{cases}
      \tau_a^{-1} & \mbox{if } s = L\\
      \tau_b & \mbox{if } s = R.\\
    \end{cases}
  \end{displaymath}
  Then
  \begin{displaymath}
    A_f = \tau(s_1) \circ \cdots \circ \tau(s_n).
  \end{displaymath}
  (By the usual convention for composition of functions, the rotations
  are applied in right-to-left order.)
\end{lemma}
\begin{proof}
  For all $1 \le k \le n$, there is a unique $\phi_k \in \Isom^+(\F)$
  that maps $e_0$ to $e_k$ and $a$ to the endpoint of $e_k$ lying on
  the left hand side of $\gamma$. We have $\phi_n = A_f$, so we need
  to prove that $\phi_n = \tau(s_1)\circ \cdots \circ \tau(s_n)$. We
  will prove by induction that
  $\phi_k = \tau(s_1)\circ \cdots \circ \tau(s_k)$ for all
  $1 \le k \le n$.

  For $k =1$, we can easily see that the isometry
  mapping $e_0$ to $e_1$ is $\tau_a^{-1}$ or $\tau_b$, depending on
  whether the common vertex of $e_0$ and $e_1$ is $a$ or $b$.

  Now assume that the claim is true for $k$ where $1\le k < n$, that
  is, $\phi_k = \tau(s_1)\circ \cdots \circ \tau(s_k)$. We want to
  show that $\phi_{k+1} = \phi_k \circ \tau(s_{k+1})$. Note that the
  edges $e_0$ and $\tau(s_{k+1})(e_0)$ of $t_1$ meet on the same side
  of $\gamma$ as the edges $e_k$ and $e_{k+1}$ of $t_{k+1}$. Since we
  have $\tau(s_k)(e_0) = e_k$ by the induction hypothesis, this
  implies
  \begin{displaymath}
    \phi_k(\tau(s_{k+1})(e_0)) = e_{k+1}.
  \end{displaymath}
  The right hand side can also be written as
  $\phi_{k+1}(e_0)$, therefore 
  $\phi_k \circ \tau(s_{k+1})$ and $\phi_{k+1}$ map $e_0$ to the same
  edge, and their actions on the endpoints also agree. Hence we have
  $\phi_{k+1} = \phi_k \circ \tau(s_{k+1})$ as claimed.
\end{proof}

\paragraph{Proof of Penner's conjecture for $S_{1,0}$, $S_{1,1}$ and
  $S_{0,4}$} Let $S$ be one of these three surfaces and let $f$ be
any pseudo-Anosov element of $\Mod(S)$. We want to show that $f$ has a
power arising from Penner's construction. 

Let $A_f$, $\gamma$, $e_0$, $a$ and $b$ be as above. Choose $a$ and $b$ for
the role of the filling curves in the construction. By Lemma
\ref{lemma:isometry_as_product}, some product of $T_a^{-1}$ and $T_b$
defines an element $h\in \Mod(S)$ such that $A_f = A_h$. Since $A_h$
is hyperbolic, both Dehn twists must appear in this product.
Therefore $h$ is a pseudo-Anosov mapping class arising from Penner's
construction.
 
When $S$ is the torus or the once-punctured torus, we have
$\Mod(S) \cong \SL(2,\Z)$. So $f = \pm h$ and hence $f^2 = h^2$. When
$S = S_{0,4}$, then $A$ is surjective with kernel
$\Z/2\Z \times \Z/2\Z$ \cite[Prop.~2.7]{FarbMargalit12}. The kernel is
generated by two hyperelliptic involutions and only its identity
element fixes all four punctures. Thus two elements of $\Mod(S_{0,4})$
that project to the same element of $\PSL(2,\Z)$ are equal if they
permute the four punctures in the same way. Therefore
$f^{12} = h^{12}$, because both maps act trivially on the punctures.
(The number 12 is the least common multiple of the orders of elements
of the symmetric group on 4 points.) Hence $f$ has a power arising
from Penner's construction.

\paragraph{Remark} Note that without allowing half-twists, the conjecture
is false for $S_{0,4}$. Indeed, an $LR$-sequence corresponding to a
product of Dehn twists is a sequence of $LL$ and $RR$ blocks. So any
pseudo-Anosov mapping class that contains the block $LRL$ in its
sequence does not have a power arising from the construction.

%
%
%
\bibliographystyle{alpha} 
\bibliography{mybibfile}

\begin{thebibliography}{FLP79}

\bibitem[AY81]{ArnouxYoccoz81}
Pierre Arnoux and Jean-Christophe Yoccoz.
\newblock Construction de diff\'eomorphismes pseudo-{A}nosov.
\newblock {\em C. R. Acad. Sci. Paris S\'er. I Math.}, 292(1):75--78, 1981.

\bibitem[Bau92]{bauer1992upper}
Max Bauer.
\newblock An upper bound for the least dilatation.
\newblock {\em Trans. Amer. Math. Soc.}, 330(1):361--370, 1992.

\bibitem[BH92]{BestvinaHandel92}
Mladen Bestvina and Michael Handel.
\newblock Train tracks and automorphisms of free groups.
\newblock {\em Ann. of Math. (2)}, 135(1):1--51, 1992.

\bibitem[CB88]{CassonBleiler88}
Andrew~J. Casson and Steven~A. Bleiler.
\newblock {\em Automorphisms of surfaces after {N}ielsen and {T}hurston},
  volume~9 of {\em London Mathematical Society Student Texts}.
\newblock Cambridge University Press, Cambridge, 1988.

\bibitem[Far06]{Farb06}
Benson Farb, editor.
\newblock {\em Problems on mapping class groups and related topics}, volume~74
  of {\em Proceedings of Symposia in Pure Mathematics}.
\newblock American Mathematical Society, Providence, RI, 2006.

\bibitem[FLP79]{FLP}
{\em Travaux de {T}hurston sur les surfaces}, volume~66 of {\em Ast\'erisque}.
\newblock Soci\'et\'e Math\'ematique de France, Paris, 1979.
\newblock S{\'e}minaire Orsay, With an English summary.

\bibitem[FM12]{FarbMargalit12}
Benson Farb and Dan Margalit.
\newblock {\em A primer on mapping class groups}, volume~49 of {\em Princeton
  Mathematical Series}.
\newblock Princeton University Press, Princeton, NJ, 2012.

\bibitem[Hir10]{Hironaka10}
Eriko Hironaka.
\newblock Small dilatation mapping classes coming from the simplest hyperbolic
  braid.
\newblock {\em Algebr. Geom. Topol.}, 10(4):2041--2060, 2010.

\bibitem[Kra81]{Kra81}
Irwin Kra.
\newblock On the {N}ielsen-{T}hurston-{B}ers type of some self-maps of
  {R}iemann surfaces.
\newblock {\em Acta Math.}, 146(3-4):231--270, 1981.

\bibitem[Lei04]{Leininger04}
Christopher~J. Leininger.
\newblock On groups generated by two positive multi-twists: {T}eichm\"uller
  curves and {L}ehmer's number.
\newblock {\em Geom. Topol.}, 8:1301--1359 (electronic), 2004.

\bibitem[Lon85]{Long85}
Darren~D. Long.
\newblock Constructing pseudo-{A}nosov maps.
\newblock In {\em Knot theory and manifolds ({V}ancouver, {B}.{C}., 1983)},
  volume 1144 of {\em Lecture Notes in Math.}, pages 108--114. Springer,
  Berlin, 1985.

\bibitem[LT11a]{LanneauThiffeault11Braids}
Erwan Lanneau and Jean-Luc Thiffeault.
\newblock On the minimum dilatation of braids on punctured discs.
\newblock {\em Geom. Dedicata}, 152:165--182, 2011.
\newblock Supplementary material available online.

\bibitem[LT11b]{LanneauThiffeault11}
Erwan Lanneau and Jean-Luc Thiffeault.
\newblock On the minimum dilatation of pseudo-{A}nosov homeromorphisms on
  surfaces of small genus.
\newblock {\em Ann. Inst. Fourier (Grenoble)}, 61(1):105--144, 2011.

\bibitem[Pen88]{Penner88}
Robert~C. Penner.
\newblock A construction of pseudo-{A}nosov homeomorphisms.
\newblock {\em Trans. Amer. Math. Soc.}, 310(1):179--197, 1988.

\bibitem[Pen91]{penner1991bounds}
Robert~C. Penner.
\newblock Bounds on least dilatations.
\newblock {\em Proc. Amer. Math. Soc.}, 113(2):443--450, 1991.

\bibitem[Shi]{Shin14}
Hyunshik Shin.
\newblock Algebraic degrees of stretch factors in mapping class groups.
\newblock Submitted.

\bibitem[Thu88]{Thurston88}
William~P. Thurston.
\newblock On the geometry and dynamics of diffeomorphisms of surfaces.
\newblock {\em Bull. Amer. Math. Soc. (N.S.)}, 19(2):417--431, 1988.

\bibitem[Zhi95]{Zhirov95}
A.~Yu. Zhirov.
\newblock On the minimum dilation of pseudo-{A}nosov diffeomorphisms of a
  double torus.
\newblock {\em Uspekhi Mat. Nauk}, 50(1(301)):197--198, 1995.

\end{thebibliography}
%



\end{document}